%% file: lost-bp3.tex
\documentclass[12pt]{amsart}
\usepackage{graphicx}
\usepackage{amssymb}
\usepackage{epstopdf}
\usepackage{amsmath,amscd}
\usepackage{amsthm}
\usepackage{url,verbatim}
\usepackage{stmaryrd}
\usepackage{caption,subcaption}

\usepackage[dvipsnames]{xcolor}

\RequirePackage[colorlinks,citecolor=blue,urlcolor=blue]{hyperref}
\usepackage{breakurl}
\theoremstyle{plain}
\newcounter{qncount}
\newcounter{solncount}
\newtheorem{theorem}{Theorem}

\newtheorem{remark}[theorem]{Remark}
\newtheorem{question}[qncount]{Question}

\newtheorem{solution}[solncount]{Solution}

\newcommand\PP{{\mathbb P}}
\renewcommand\Pr{\PP}

\newcommand\qq{\qquad}
\newcommand\oo{\infty}
\newcommand\resp{respectively}

\newcounter{mycount1}\newcounter{mycount2}\newcounter{mycount3}\newcounter{mycount}

\newenvironment{numlist}{\begin{list}{\rm\arabic{mycount2}.}%
   {\usecounter{mycount2}\labelwidth=1cm\itemsep 0pt}}{\end{list}}

\numberwithin{equation}{section}
\numberwithin{theorem}{section}
\numberwithin{figure}{section}

\title[The lost boarding pass, and other practical problems]{The lost boarding pass,\\ and other practical problems}
\author{Geoffrey R. Grimmett, David R. Stirzaker}
\address{Statistical Laboratory, Centre for
Mathematical Sciences, Cambridge University, Wilberforce Road,
Cambridge CB3 0WB, UK} 
\email{g.r.grimmett@statslab.cam.ac.uk}
\urladdr{\url{http://www.statslab.cam.ac.uk/~grg/}}
\address{St John's College,
Oxford OX1 3JP, UK}
\email{david.stirzaker@sjc.ox.ac.uk}
\urladdr{\url{https://www.sjc.ox.ac.uk/discover/people/professor-david-stirzaker/}}

\begin{document}
\begin{abstract}
The reader is reminded of several puzzles involving randomness. These may be ill-posed,
and if well-posed there is sometimes a solution that uses probabilistic intuition in
a special way. Various examples are presented including the well known
problem of the lost boarding pass: what is the probability that the last passenger
boarding a fully booked plane sits in the assigned seat if the first passenger has occupied a randomly chosen seat?
This problem, and its striking answer of $\frac12$, has attracted  a good deal of attention since around 2000.
We review elementary solutions to this, and to the more general problem of finding the probability
the $m$th passenger sits in the assigned seat when in the presence of some number $k$ of passengers with lost
boarding passes. A simple proof is presented of the independence of the occupancy status of different seats,
and a connection to the Poisson--Dirichlet distribution is mentioned. 
\end{abstract}

\date{October 6, 2019} 

\maketitle

\section{Puzzles in probability}

The world is amply supplied with probability teasers. These can mystify and intoxicate. Here are some 
classics\footnote{All problems in this article  are considered in some detail in Grimmett
and Stirzaker \cite[Chaps 1, 3, 4, 12]{otep2e, prp3e}.}.
\begin{numlist}
\item {(Family planning)} A family has three children, two of whom are boys. What is the probability the third is a boy?
\item {(Bertrand's paradox)} A chord of the unit circle is picked at random. What is the
probability that an equilateral triangle with the chord as base can fit inside the circle?
\item {(Monty Hall problem)} Monty Hall shows you three doors concealing, \resp, two goats and one Cadillac, and he asks you
to choose a door and take whatever is behind it. Without opening your door, he opens another and displays a goat. 
Is it to your advantage to accept his invitation to shift your choice to the third door?
\end{numlist}
These three questions have the common feature of being \lq ill posed': insufficient information
is given for the answer to be properly determined.

There are further teasers that tempt the reader to indulge in sometimes over-complex calculations. Here is one such.

\begin{numlist}
\item [4.] {(Red Now)}
This game is played by a single player with
a well shuffled conventional pack of 52 playing cards. At times $n=1,2,\dots,52$
the player turns over a new card and observes its colour.
Exactly once in the game s/he must say, just before exposing a
card, ``Red Now''. The game is won if the next exposed card is red.
Is there a strategy
for the player that results in a probability of winning different
from $\frac12$?
\end{numlist}

It will be tempting to embark on a calculation that analyses the success 
probability of a general strategy, possibly using
some fairly heavyweight machinery such as optional stopping. 
One would need firstly to decide what exactly makes  a \lq strategy', and secondly 
which strategies to prioritise (such as waiting until
there is a given preponderance of reds in the unseen cards).
There is however a beautiful 
\lq one-line' solution\footnote{Explained
to one of the authors by an undergraduate following a lecture.}
that finesses such complexities as follows,
at the price of employing a touch of what might be called \lq probabilistic intuition'.
It may be seen that the chance of winning is the same
for a player who, after calling “Red Now”, picks the card placed at the bottom of the pack rather than
that at the top. At the start of this game, 
the bottom card is red with probability $\frac12$ irrespective of the strategy of the player.

We amplify this discussion with a more detailed analysis in Sections \ref{sec:int}--\ref{sec:many}
of the well known \lq lost boarding pass problem'\footnote{See \cite[Prob.\ 1.8.39]{otep2e, prp3e}.}. 
This provides a
surprising illustration of the probabilistic property of \lq independence' of events, and
reveals the full structure of the random mapping of passengers to seats.

\section{One lost boarding pass}\label{sec:int}

\begin{question}\label{q1}
(Lost boarding pass problem) A plane has $n\ge 2$ seats, and is fully booked. 
A line of $n$ people board one by one. Each has an assigned seat, but the first passenger 
(doubtless a mathematician thinking of other things) has
lost his or her boarding pass and sits in a seat chosen at random.  On boarding,
subsequent passengers sit in their assigned seats if free, 
and in randomly chosen empty seats if not. What is the probability that
the last passenger sits in his or her assigned seat?
\end{question}

\begin{figure}
\includegraphics[width=0.9\textwidth]{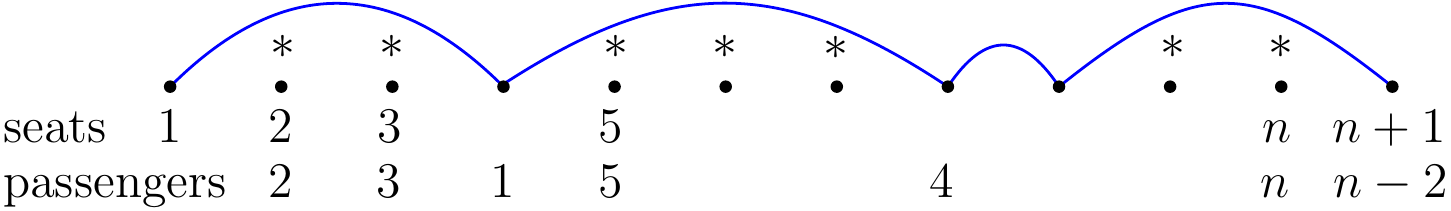}
\caption{Seats are labelled $1,2,\dots,n$, and the extra label $n+1$ represents a copy of seat $1$.
In this example, passenger $1$ sits in seat $4$, passenger $4$ is displaced into a later seat, and the chain of
displacement continues until some displaced passenger chooses seat $n+1$ ($\equiv 1$).
An asterisk denotes a correctly seated passenger.}\label{fig:1}
\end{figure}

This problem has generated a good deal of discussion.
It is well known that the answer is $\frac12$, and several arguments may be found  for this
in puzzle books and online (see, for example, \cite{stack}). 
The neatest is probably the following 
(attributed by Peter Winkler \cite[p.\ 35]{Wink} in 2004
to Alexander Holroyd, but perhaps known contemporaneously to others). Let the passengers be labelled
$1,2,\dots,n$ in order of boarding, and label their assigned seats similarly (so that passenger $k$ has
been assigned the seat labelled $k$). The seating process is illustrated in
Figure \ref{fig:1}.

\begin{solution}\label{soln1}
When the $n$th passenger chooses
a seat, there is only one available, and this must be either seat $1$ or seat $n$
(since seat $r\in\{2,3,\dots,n-1\}$, if free, would have been claimed earlier by passenger $r$). 
Each of these two possibilities has probability $\frac12$, since no earlier decision has distinguished between them. 
\end{solution}

Some will be uncomfortable with this solution, in that it uses a symmetry rather than an equation.
Such is not atypical of arguments considered by mathematicians to be \lq neat', though they 
frequently have the potential to mystify
the uninitiated (see, for example, \cite[p.\ 176]{bb}).  

The solution presented in \cite{otep2e} is more prosaic than the above,
and utilizes
conditional probability to derive a recurrence relation. Such a method was used by Dermot Roaf \cite{roaf}
to obtain an answer to the following elaboration of Question \ref{q1}.

\begin{question}\label{q2}
What is the probability that passenger $m$ finds seat $m$ to be already occupied?
\end{question}

We do not include Roaf's solution (which uses an equation). Here however is the \lq one-liner'\footnote{Possibly
noted first by Nigussie \cite{Nigussie}.} that extends Solution \ref{soln1}. 

\begin{solution}\label{soln2}
Let $2\le m\le n$. 
When
passenger $m$ chooses a seat, seats $2,3,\dots,m-1$ are already taken (since seat $r\in\{2,3,\dots,m-1\}$, 
if free, would have been claimed earlier by passenger $r$), and one further seat also. 
No information is available about the label of this further seat.
By symmetry, it is equally likely to be any of the $n-m+2$ seats labelled
$1,m,m+1,\dots,n$. Therefore, it is seat $m$
with probability $1/(n-m+2)$.
\end{solution}

Let $D_m$ be the event that passenger $m$ finds seat $m$
already occupied. 
Henze and Last \cite{henze} have
observed that the $D_m$ are independent events. We explain this using the 
simple argument of Solution \ref{soln2}.

\begin{theorem}\label{thm0}
The events $D_2,D_3,\dots,D_n$ are independent.
\end{theorem}

\begin{proof}
This is by iteration of the argument of Solution \ref{soln2}. Let $2\le m_1<m_2\le n$, so that
\begin{equation}\label{eq:7}
\Pr(D_{m_1}\cap D_{m_2}) = \Pr(D_{m_1}) \Pr(D_{m_2}\mid D_{m_1}).
\end{equation}
When passenger $m_1$ chooses a seat, seats $2,3,\dots,m_1-1$ have already been taken, 
together with one further seat
denoted $S_1$. Conditional on $D_{m_1}$, we have $S_1=m_1$, so that passenger $m_1$ is 
displaced to some later seat. When, subsequently, passenger $m_2$ chooses a seat, the seats $m_1+1,m_1+2,\dots,m_2-1$
are occupied and in addition one further seat denoted $S_2$.  The choices made so far
contain no information about $S_2$ other than it is
one of $1,m_2,m_2+1,\dots,n$.
Thus, $\Pr(D_{m_2}\mid D_{m_1})$
equals the (conditional) probability that $S_2=m_2$,
namely $1/(n-m_2+2)$. By \eqref{eq:7},
\begin{equation}\label{eq:8}
\Pr(D_{m_1}\cap D_{m_2}) = \frac1{n-m_1+2}\cdot\frac 1{n-m_2+2}=\Pr(D_{m_1})\Pr(D_{m_2}),
\end{equation}
so that $D_{m_1}$ and $D_{m_2}$ are independent. By the same argument iterated,
any finite subset of  $\{D_m: 2\le m\le n\}$ is independent.
\end{proof}

We summarise the above as follows. By Solution \ref{soln2} and Theorem \ref{thm0}, the $D_m$ are independent wth probabilities $p_m:=\Pr(D_m)$ given by
\begin{equation}\label{eq:9}
p_m=\frac 1{n-m+2},\qq m=2,3,\dots,n.
\end{equation}

We may construct the seating arrangement \lq backwards'. 
For $s=1,2,\dots,n$, we declare $s$ to be \emph{red} with probability
$1/s$, and \emph{black} otherwise, with different integers having independent colours.
We relabel seat $1$ as seat $n+1$, so that the seat-set is
now  $S=\{2,3,\dots,n+1\}$, and we declare seat $m$ to be red if the integer $s=n-m+2$ is red,
which occurs with probability $p_m$. The  passengers may now be seated. Passenger $1$ sits in the earliest red seat,
$R_1$ say; passenger $R_1$ sits in the next red seat, $R_2$ say, and so on; undisplaced passengers
sit in their assigned seats. 

\begin{remark}\label{rem1}
The harmonic progression of \eqref{eq:9} provokes a coupling of seat colours inspired by the theory of record values.
Let $\{U_s: 1\le s\le n\}$ be independent random variables with the uniform distribution on $[0,1]$.
The integer $s$ is called a \emph{record value} if $U_s>U_r$ for $r<s$. 
It is standard\/\footnote{See, for example, \cite[Prob.\ 7.11.36]{otep2e,prp3e}.} that the 
events $V_s=\{\text{\rm$s$ is a record value}\}$, for $1\le s \le n$,
are independent with $\Pr(V_s)=1/s$. The seats may be coloured as follows.
Seat $m \in S$ is coloured red if and only if $n-m+2$ is a record value.
\end{remark}

\begin{remark}\label{rem:2}
What happens in the limit as $n\to\oo$, and in particular how does Figure \ref{fig:1} evolve when scaled by 
the factor $1/n$?
Let $S_i$ be the seat occupied by passenger $i$, with the convention that the seat labelled $1$ is relabelled $n+1$
(see Figure \ref{fig:1}).
Note that $S_i \ge i$ for all $i$, and $S_i=i$ if and only if $i \ne 1$ and $i$ is correctly seated.
The displacements $D_i:= S_i-i$ are a measure of the degree of disorder in the plane.
It turns out that the normalised vector $D=(D_1/n,D_2/n,\dots,D_n/n)$, when re-ordered in
decreasing order (so that the first entry is the maximum displacement),  has  a limiting distribution
in the limit as $n \to\oo$. The limit is the so-called Poisson--Dirichlet distribution, which crops up in 
a variety of settings including probabilistic number theory, the theory of random permutations,
and population genetics\/\footnote{See Terry Tao's blog \cite{tao2013}.}. 
The key reason for this is illustrated in Figure \ref{fig:1}: $S_1$ is uniformly distributed on the set of seats;
passenger $S_1$ sits in a seat labelled $S_2$ that is uniform on the seats to the right of $S_1$; and so on.
  
\end{remark}

Much mathematics proceeds by identification and solution of a sequence of smaller problems, which are obtained 
from the original by  a process of distillation
of intrinsic difficulties. In this way, even teasers may play significant roles in solving challenging
problems.  The lost boarding pass problem 
(more specifically, the argument of Solutions \ref{soln1} and \ref{soln2})
has contributed to work by Johan W\"astlund \cite{wast}
on the random assignment problem and the so-called Buck--Chan--Robbins urn process. 

\section{Many lost boarding passes}\label{sec:many}

A more general version of Questions \ref{q1} and \ref{q2} has been studied by Henze and Last \cite{henze}.

\begin{question}\label{q3}
Let $k \ge 1$. A plane has $n\ge k+1$ seats, and is fully booked. 
A line of $n$ people board one by one. Each has an assigned seat, but the first $k$ passengers 
have
lost their boarding passes and sit in seats chosen at random from those available. 
Subsequent passengers sit in their assigned seats if free, 
and in randomly chosen empty seats if not. Let $k+1 \le m\le n$.
What is the probability  that
passenger $m$ finds seat $m$ already occupied?
\end{question}

\begin{figure}
\includegraphics[width=0.9\textwidth]{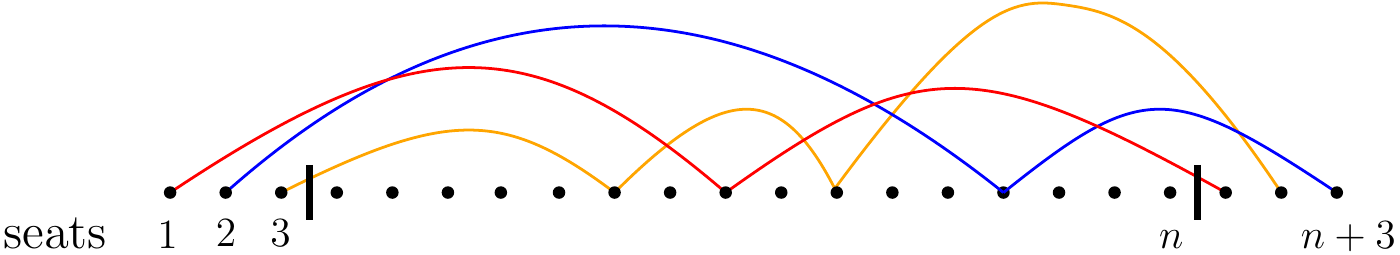}
\caption{With the first $3$ boarding passes lost, each of the first $3$ passengers
initiates a chain of displacements.}\label{fig:2}
\end{figure}

The solution using symmetry follows 
the route of Solution \ref{soln2}.

\begin{solution}\label{soln3}
Let $k+1\le m\le n$. 
When
passenger $m$ chooses a seat, the seats $k+1,k+2,\dots,m-1$ are already  taken (since seat $r\in\{k+1,k+2,\dots,m-1\}$, 
if free, would have been claimed earlier by passenger $r$), and $k$ further seats also. 
By symmetry, these further seats are equally likely to be any $k$-subset of the $n-m+k+1$ seats labelled
$1,2,\dots,k,m,m+1,\dots,n$. Therefore, seat $m$ is occupied with probability $k/(n-m+k+1)$.
\end{solution}

In particular, the probability that passenger $n$ sits in his or her assigned seat is 
$$
1-\frac k{k+1}= \frac 1{k+1}.
$$
We turn finally to independence.
As in Section \ref{sec:int}, let $D_m$ be the event that passenger $m$ finds seat $m$ already occupied.

\begin{theorem}\label{thm3}
The events $D_{k+1},D_{k+2},\dots,D_n$ are independent.
\end{theorem}

\begin{proof}
This follows that of Theorem \ref{thm0}. For example, by the same argument, \eqref{eq:8} becomes
\begin{align*}
\Pr(D_{m_1}\cap D_{m_2}) &= \frac k{n-m_1+k+1}\cdot\frac k{n-m_2+k+1}\\
&=\Pr(D_{m_1})\Pr(D_{m_2}),
\end{align*}
and similarly for intersections of three or more of the events $D_m$.
\end{proof}

The overall seating program is now as follows, using the language of the end of Section \ref{sec:int}. 
We have $S=\{k+1,k+2,\dots,n,n+1,\dots,n+k\}$, and seat $m\in S$ is declared \emph{red} with
probability 
$$
p_m= \begin{cases} \dfrac k{n-m+k+1} &\text{if } m \le n,\\
1 &\text{if } m>n,
\end{cases}
$$
and \emph{black} otherwise, different seats being coloured independently of one another. 
There are $k$ distinct shades $\rho_1,\rho_2,\dots,\rho_k$ of red. Each red seat $m$ satisfying $m\le n$ 
is shaded with a uniformly random shade, chosen
independently of the shades of other red seats.  As for the seats labelled $n+1,n+2,\dots,n+k$, for
each $i$ exactly one 
of these is shaded 
$\rho_i$, with the distribution of shades being given by a random permutation of the labels. 

The passengers may now be seated according to the following algorithm.
Passenger $1$ sits in the earliest seat shaded $\rho_1$; the displaced passenger 
sits in the next seat shaded $\rho_1$, and so on.
The process is repeated for passengers $2,3,\dots,k$, and finally
the black seats are occupied by their correct passengers.

\begin{remark}\label{rem:3}
In Remark \ref{rem:2} we identified the limit of the seating process as $n\to\oo$, with displacements
scaled by $1/n$. With $k \ge 1$ lost boarding passes, the diagram of Figure \ref{fig:2}
converges, when rescaled,  to $k$ independent copies of that of Figure \ref{fig:1}.
\end{remark}

One may also consider the more general situation when the lost boarding passes may not be consecutive.
The associated probabilities can be calculated using the methods and conclusions given above.

\section*{Acknowledgement}
We are grateful to David Bell for having drawn our attention to Question \ref{q1}
in 2000 or earlier.

\input{lost-bp3.bbl}


\end{document}

%% file: lost-bp3.bbl
\providecommand{\bysame}{\leavevmode\hbox to3em{\hrulefill}\thinspace}
\providecommand{\MR}{\relax\ifhmode\unskip\space\fi MR }
\providecommand{\MRhref}[2]{%
  \href{http://www.ams.org/mathscinet-getitem?mr=#1}{#2}
}
\providecommand{\href}[2]{#2}